\documentclass{amsart}
\usepackage[top=25mm, bottom=25mm, left=25mm, right=25mm]{geometry}
\usepackage{graphicx}
\usepackage{amsmath}
\allowdisplaybreaks[4]
\usepackage{amssymb}
\usepackage{amsthm}
\usepackage{amsfonts}
\usepackage{bm}
\usepackage{mathtools}
\usepackage{empheq}
\usepackage{tikz}

\theoremstyle{definition}

\newtheorem{definition}{Definition}[section]

\newtheorem{theorem}[definition]{Theorem}
\newtheorem{corollary}[definition]{Corollary}
\newtheorem{remark}[definition]{Remark}

\numberwithin{equation}{section}

\newcommand{\mQ}{\mathbb{Q}}
\newcommand{\mZ}{\mathbb{Z}}
\newcommand{\mF}{\mathbb{F}}
\newcommand{\mFp}{\mathbb{F}_{p}}
\newcommand{\wi}{\widetilde}
\newcommand{\ov}{\overline}

\title{Torsion Packet Envelope and Rational Points of Algebraic Curves}
\author{Ryo Ichikawa}
\date{}

\keywords{Rational points; Algebraic curves; Hyperelliptic curve}
\subjclass[2010]{11G30}
\address{Ryo Ichikawa \\ 
Mathematical Inst. Tohoku Univ.\\
 6-3,Aoba, Aramaki, Aoba-Ku, Sendai 980-8578, JAPAN}
\email{ichikawa.ryo.t8@dc.tohoku.ac.jp}

\begin{document}

\begin{abstract}\noindent
  In this paper, we give an elementary new method for determining the rational points on algebraic curves using torsion packets. We also provide examples of curves for which all rational points can be completely determined by our method.
\end{abstract}
\maketitle
\tableofcontents

\section{Introduction}
Determining all rational points on a given algebraic curve is a fundamental problem in number theory. Faltings proved in \cite{Faltings1983} that a smooth algebraic curve over a number field $K$ of genus at least $2$ has only finitely many $K$-rational points. However, his proof is ineffective in the sense that it does not yield an explicit upper bound on the number of rational points on a given curve. To determine all rational points on a specific algebraic curve, it is often necessary not only to search for rational solutions of the defining equation, but also to establish explicit bounds restricting the possible rational points.
In this direction, the works \cite{Coleman1985}, \cite{Stoll2006}, and \cite{KZ2013} establish explicit upper bounds via the Coleman integral. These methods, however, provide only bounds on the number of rational points and do not, in general, determine them completely.
In the present paper, we generalize the method of \cite{T-Y2008} and introduce a new approach that avoids the use of Coleman integrals. More precisely, we construct a torsion packet containing the set of rational points. As an application, we present examples of algebraic curves for which all rational points can be explicitly determined by this method.

We fix some notation to state our results. Let $C$ be a smooth geometrically connected projective curve of genus $g\ge 2$ over $\mQ$. Let $J$ denote the Jacobian of $C$. We assume that $C(\mQ)\neq \emptyset$. Fix a point $P_{0}\in C(\mQ)$. We define a morphism
\begin{equation*}
    \iota : C(\ov{\mQ})\rightarrow J(\ov{\mQ}),P\mapsto P-P_{0}.
\end{equation*}
Note that this morphism is defined over $\mQ$, and it is injective since $g\ge 1$. We consider triples $(F,T,w)$ that satisfy the following conditions:

\begin{enumerate}
    \item $F$ is a finite extension of $\mQ$;\label{eq:TPEa}
    \item $w$ is a finite place of $F$ lying over a prime $p$ such that odd and completely split over $F$;\label{eq:TPEb}
    \item $C$ has a good reduction at $p$;\label{eq:TPEc}
    \item $T$ is a finite subset of $C(F)$ with $\iota(T)\subseteq J(F)_{\mathrm{tor}}$;\label{eq:TPEd}
    \item $\#T \ge \#\wi{C}(\mF_{w})$.\label{eq:TPEe}
\end{enumerate}

where $J(F)_{\mathrm{tor}}$ denotes the subgroup of torsion points of $J(F)$. 

\begin{definition}
    The notation being as before. The triple $(F,T,w)$ that satisfies \eqref{eq:TPEa}-\eqref{eq:TPEe} for $C$ is said to be a $\textit{torsion\ packet\ envelope}$ for $C(\mQ)$. We sometimes call it a $\textit{TPE}$.
\end{definition}

Then, our result is the following.

\begin{theorem}\label{MainThm}
    The notation being as before. Assume that there exists a torsion packet envelope $(F,T,w)$ for $C(\mQ)$. Then,  
    $C(\mQ)\cap J(\mQ)_{\mathrm{tor}}\coloneqq \iota^{-1}(\iota(C(\mQ))\cap J(\mQ)_{\mathrm{tor}})\subseteq T$.
\end{theorem} 
\begin{remark}The proof of Theorem~\ref{MainThm} shows that 
$\#T = \#\wi{C}(\mF_{w})$.
\end{remark}

In particular, we obtain the following result.

\begin{corollary}
    If $\iota(C(\mQ))\subseteq J(\mQ)_{\mathrm{tor}}$, then $C(\mQ)\subseteq T$. In particular, if $\operatorname{rank}J(\mQ)=0$, then $C(\mQ)\subseteq T$.
\end{corollary}

The notation being as before. As an application, we apply our method to the following families of curves and explicitly determine their rational points.

\begin{corollary}\label{compute}
\leavevmode\par
    \begin{enumerate}
        \item Let $C_{d}$ be the hyperelliptic curve $y^2=x^5+d$ where $d\in \mZ \setminus \{0\}$ is tenth-power free and $d\equiv 1,7,9 \pmod{11}$. If $\operatorname{rank}J(\mQ)=0$, then
        \begin{equation*}
            C_{d}(\mQ)=
            \begin{cases}
                \{ P_{\infty} \} & \text{if $\ov{d}=\ov{7}$ }, \\
                \{ (0,\pm\sqrt{d}),P_{\infty} \} & \text{if $\ov{d}=\ov{9}$ and $d$ is a perfect square}, \\
                \{ P_{\infty} \} & \text{if $\ov{d}=\ov{9}$ and $d$ is not a perfect square}, \\
                \{ (0,\pm\sqrt{d}),P_{\infty} \} & \text{if $\ov{d}=\ov{1}$, $d$ is a perfect square and is not a fifth power}, \\
                \{(-\sqrt[5]{d},0),P_{\infty}\} & \text{if $\ov{d}=\ov{1}$, $d$ is not a perfect square and is a fifth power}, \\
                \{ P_{\infty} \} & \text{if $\ov{d}=\ov{1}$, $d$ is not a perfect square and is not a fifth power}.
            \end{cases}
        \end{equation*}

        Table~1 lists the values of $d$ with $|d|\le 200$ for which the Jacobian of $C_{d}$ has rank $0$ over $\mQ$, computed using Magma \cite{Magma}.
        
        \item Let $p$ be a prime with $p\equiv 3 \pmod{4}$. Let $D_{d}$ be the hyperelliptic curve $y^2=f(x)\coloneqq x^{p-1}+dx^{\frac{p-1}{2}}-1$ where $d\in p\mZ$. If $\operatorname{rank}J(\mQ)=0$, then
        \begin{equation*}
            D_{d}(\mQ)\subseteq \{ (\alpha,0),P_{\pm\infty}\ \mid\ f(\alpha)=0, \alpha\in\ov{\mQ} \}.
        \end{equation*}
        For example, using Magma \cite{Magma}, when $p=7$ and $0\leq d \leq 100$, the following values of $d$ satisfy $\operatorname{rank}J_{D_{d}}(\mQ)=0$:
        \begin{equation*}
            d\in\{ 0,42,70,98 \}.
        \end{equation*}
        
        \item Let $p$ be a prime with $p\ge 5$ and let $C$ be the hyperelliptic curve $y^2=x^p-x$. If $\operatorname{rank}J(\mQ)=0$, then
        \begin{equation*}
            C(\mQ)= \{ (0,0),(\pm1,0),P_{\infty} \}.
        \end{equation*}
        For example, using Magma \cite{Magma}, $\operatorname{rank}J_{C}(\mQ)=0$ for $p=5,13,17$.
    \end{enumerate}
\end{corollary}

\begin{table}
\begin{center}
\caption{}
\renewcommand{\arraystretch}{1.5}
\begin{tabular}{|c|c|} \hline
    $\ov{d}$ & Values of $d$ with $|d|\leq 200$ such that $\operatorname{rank}J_{C_{d}}(\mQ)=0$ \\ \hline
    $\ov{7}$ & $-191, -180, -169, -158, -147, -136, -125, -114, -103, -92, -70, -59, $ \\ 
    & $-37, -26, -15, -4, 18, 29, 51, 62, 84, 95, 106, 139, 161, 172, 183$ \\ \hline
    $\ov{9}$ & $-200, -189, -178, -167, -145, -134, -123, -90, -68, -57, -46, -35, -13, -2, $ \\
    & $ 31, 42, 53, 64, 75, 86, 97, 119, 130, 141, 152, 163, 174, 185, 196$ \\ \hline
    $\ov{1}$ & $-186, -164, -153, -142, -131, -120, -109, -98, -76, -65, -54, -21, -10, $ \\
    & $ 1, 12, 23, 56, 67, 78, 100, 111, 133, 144, 155, 166, 188$ \\ \hline
\end{tabular}
\end{center}
\end{table}

The proof of Theorem~\ref{MainThm} is inspired by Section~7 of \cite{T-Y2008}. While \cite{Coleman1985} and \cite{Stoll2006} obtain explicit upper bounds using Coleman integration, in the present work we obtain a finite set containing all rational points by a more elementary approach. An apparent advantage of our method is that the statement of Theorem~\ref{MainThm} is stated independently of the rank of the Jacobian over an extension field though we need to check the fourth condition of TPE. In concrete applications, such as Corollary~\ref{compute}, one can combine this result under the assumption that the Jacobian over $\mQ$ has rank $0$. We hope that our result may have potential application to a certain general theory to attach rational points on algebraic curves over number fields.

We organize the paper as follows. In Section~2, we give a proof of our main theorem. In Section~3, we give some examples of curves for which all rational points can be determined explicitly by our main theorem.

\subsection*{Acknowledgements}
The author would like to express his sincere gratitude to his supervisor, Professor Takuya Yamauchi, for valuable advice and constant encouragement. The author would also like to thank Daichi Tanaka for many helpful comments and valuable advice on writing this paper.

\section{Proof of Main Theorem}
In this Section, we give a proof of Theorem~\ref{MainThm}.
\begin{proof}
We consider the following commutative diagram: 
\begin{equation*}
\begin{tikzpicture}[auto]
\node (a) at (2.4, 4.8) {$T$};
\node (b) at (4.8, 4.8) {$J(F)_{\mathrm{tor}}$};
\node (c) at (0, 2.4) {$C(\mQ)\cap J(\mQ)_{\mathrm{tor}}$}; 
\node (d) at (2.4, 2.4) {$C(F)$};
\node (e) at (4.8, 2.4) {$J(F)$};
\node (f) at (0, 0) {$\wi{C}(\mFp)$};
\node (g) at (2.4, 0) {$\wi{C}(\mF_{w})$};
\node (h) at (4.8, 0) {$\wi{J}(\mF_{w})$};
\draw[->] (a) to node {$\iota$} (b);
\draw[->] (c) to node {$\subseteq$} (d);
\draw[->] (d) to node {$\iota$} (e);
\draw[->] (f) to node {$=$} (g);
\draw[->] (g) to node {$\wi{\iota}$} (h);
\draw[->] (a) to node[xshift=-8pt, yshift=5pt, rotate=270] {$\subseteq$} (d);
\draw[->] (b) to node[xshift=-8pt, yshift=5pt, rotate=270] {$\subseteq$} (e);
\draw[->] (c) to node {red} (f);
\draw[->] (d) to node {red} (g);
\draw[->] (e) to node {red} (h);
\draw[->] (c) to[bend left=60] node {$\iota$} (b);
\end{tikzpicture}
\end{equation*}

where the vertical arrows to the bottom line stand for the reduction maps to the good reductions $\wi{C}$ and $\wi{J}$, respectively. We focus on the composition $J(F)_{\mathrm{tor}}\rightarrow J(F)\rightarrow \wi{J}(\mF_{w})$. Let $J(F)_{\mathrm{tor}}^{(p)}$ denote the subgroup of elements of $J(F)_{\mathrm{tor}}$ of order prime to $p$, and let $J(F)[p^{\infty}]$ denote the subgroup of torsion points in $J(F)$ of $p$-power order. Consider the restrictions of this map to

\begin{align}
    J(F)_{\mathrm{tor}}^{(p)}\rightarrow \wi{J}(\mF_{w}), \label{eq:tor} \\
    J(F)[p^{\infty}]\rightarrow \wi{J}(\mF_{w}).\label{eq:p^{infty}}
\end{align}

Since $J(F)$ is an abelian variety with good reduction at $w$, \eqref{eq:tor} is injective by Theorem~C.1.4 in \cite{H-S2000}. Since $e_{p}=1<p-1$, \eqref{eq:p^{infty}} is injective by the appendix of \cite[see p.502 Lemma]{Katz1981}. This can be also checked as follows. 
Put $C=J(F)[p^{\infty}]$. Let $\mathcal{J}$ be the N\'eron model over $\mZ_p$ (\cite[p.16, Corollary 2]{Bosch-et-al}). Then, the scheme theoretic closure $\mathcal{C}$ of $C$ in $\mathcal{J}$ 
is finite flat over $\mZ_p$ by \cite[Section 2.1, p.259, line -1 to p.260, the first line]{Raynaud1974}. By \cite[p.268, COROLLAIRE 3.3.6.]{Raynaud1974}, $\mathcal{C}$ is determined by the constant group $C$. 
Thus, $\mathcal{C}$ is a constant finite flat group scheme over 
$\mZ_p$. Hence, it is finite \'etale. Then, the special fiber of the natural inclusion $\mathcal{C}\subset \mathcal{J}$ is the reduction map (\ref{eq:p^{infty}}) and it is injective.

Thus the composition  $J(F)_{\mathrm{tor}}=J(F)^{(p)}_{\mathrm{tor}}\times J(F)[p^{\infty}] \rightarrow J(F)\rightarrow \wi{J}(\mF_{w})$ is injective. By the commutative diagram, the composition 
\begin{equation}\label{eq:Tmor}
    T\rightarrow C(F)\rightarrow \wi{C}(\mF_{w})
\end{equation}
is injective. Since $\#T \ge \#\wi{C}(\mF_{w})$, \eqref{eq:Tmor} is bijective. 

We show that $C(\mQ)\cap J(\mQ)_{\mathrm{tor}}\subseteq T$. Let $P$ be an element of $C(\mQ)\cap J(\mQ)_{\mathrm{tor}}$. We consider $\iota(P)\in J(F)_{\mathrm{tor}}$ and $\wi{\iota(P)}\in \wi{J}(\mF_{w})$. Then $\wi{\iota}(\wi{P})=\wi{\iota(P)}$ by the commutative diagram. Since \eqref{eq:Tmor} is bijective, there exists a unique element $Q\in T$ such that $\wi{Q}=\wi{P}$. Then $\iota(Q)\in J(F)_{\mathrm{tor}}$ satisfies $\wi{\iota(Q)}=\wi{\iota}(\wi{Q})=\wi{\iota}(\wi{P})=\wi{\iota(P)}$, and hence $\iota(Q)=\iota(P)$, since the composition $J(F)_{\mathrm{tor}}\rightarrow J(F)\rightarrow \wi{J}(\mF_{w})$ is injective. Therefore $P=Q$, since $\iota\colon C(F)\rightarrow J(F)$ is injective. We conclude that $C(\mQ)\cap J(\mQ)_{\mathrm{tor}}\subseteq T$.

\end{proof}

\section{Examples}
In this section, by using Theorem~\ref{MainThm}, we present examples of projective curves for which the rational points can be determined explicitly. We begin with a brief review of hyperelliptic curves. Let $K$ be a number field and $\ov{K}$ be its algebraic closure.

\begin{definition}
    A $\mathit{hyperelliptic\ curve}$ $C/K$ is a smooth projective curve of genus $g\geq 1$ such that an affine model of $C$ can be written as
    \begin{equation*}
        y^2=f(x),f(x)\in K[x]
    \end{equation*}
    with $2g+1\leq \deg (f)\leq 2g+2$.
\end{definition}

We now describe the explicit construction by gluing. Suppose that $\deg f(x)=2g+1$ and
\begin{equation*}
    f(x)=a_{0}x^{2g+1}+a_{1}x^{2g}+\dots +a_{2g+1} \in K[x],\ (a_{0}\neq 0).
\end{equation*}
Then the hyperelliptic curve $C$ is obtained by gluing the affine curves
\begin{equation*}
    C_{0}:y^2=a_{0}x^{2g+1}+a_{1}x^{2g}+\dots +a_{2g+1}
\end{equation*}
and
\begin{equation*}
    C_{1}:w^2=a_{0}v+a_{1}v^2+\dots +a_{2g+1}v^{2g+2}
\end{equation*}
via the change of variables $(v,w)=(x^{-1},yx^{-(g+1)})$. We denote by $P_{\infty}$ the point $(0,0)$ on $C_{1}$. In particular, we have
\begin{equation*}
    C(K)=C_{0}(K)\cup\{P_{\infty}\}.
\end{equation*}
Next, suppose $\deg f(x)=2g+2$ and
\begin{equation*}
    f(x)=a_{0}x^{2g+2}+a_{1}x^{2g+1}+\dots +a_{2g+2},\  (a_{0}\neq 0).
\end{equation*}
In this case, the hyperelliptic curve $D$ over $K$ is obtained by gluing the affine curves
\begin{equation*}
    D_{0}:y^2=a_{0}x^{2g+2}+a_{1}x^{2g+1}+\dots +a_{2g+2}
\end{equation*}
and
\begin{equation*}
    D_{1}:w^2=a_{0}+a_{1}v+\dots +a_{2g+2}v^{2g+2}
\end{equation*}
via the change of variables $(v,w)=(x^{-1},yx^{-(g+1)})$. We denote the points on $D_{1}$ by
\begin{equation*}
    P_{+\infty}=(v,w)=(0,\sqrt{a_{0}}),\ P_{-\infty}=(v,w)=(0,-\sqrt{a_{0}})
\end{equation*}
where $\sqrt{a_{0}}$ is a square root of $a_{0}$ in $\ov{K}$. If $\sqrt{a_{0}}\in K$, we have
\begin{equation*}
    D(K)=D_{0}(K)\cup\{ P_{\pm\infty} \}.
\end{equation*}

\begin{definition}
    Let $\tau$ denote the \textit{hyperelliptic involution} $\tau\colon (x,y)\mapsto (x,-y)$. The \textit{Weierstrass points} $P_{1},\dots ,P_{2g+2}$ are the $\ov{K}$-rational fixed points of $\tau$. A point that is not a Weierstrass point is called a \textit{non-Weierstrass point}.
\end{definition}

\begin{remark}
    If $\deg f(x)=2g+1$, then $P_{\infty}$ is a Weierstrass point. If $\deg f(x)=2g+2$, then $P_{\pm\infty}$ are non-Weierstrass points.
\end{remark}

\begin{remark}
    Let $C\colon y^2=f(x)$ be a hyperelliptic curve over a number field $K$, and let $J$ be the Jacobian of $C$. If $\deg f(x)=2g+1$, then the images of the Weierstrass points under $\iota$ belong to $J(\ov{K})[2]$. If $\deg f(x)=2g+2$, then the images of the Weierstrass points under $\iota$ belong to $J(\ov{K})[2]$, and $\iota(P_{\pm\infty})\in J(\ov{K})[2]$. These results follow from Proposition~2.1 of \cite{Sairaiji2000}.
\end{remark}

\subsection{The hyperelliptic curve $y^2=x^5+d$ }
Let $C_{d}$ be the hyperelliptic curve $y^2=x^5+d$, where $d \in \mZ \setminus \{0\}$ is tenth-power free. Fix a prime $p=11$. We have
\begin{equation*}
    \#\wi{C}(\mF_{11})=
    \begin{cases}
        1 & \text{if $\ov{d}=\ov{7}$}, \\
        3 & \text{if $\ov{d}=\ov{9}$}, \\
        8 & \text{if $\ov{d}=\ov{1}$}, \\
        \geq 11 & \text{otherwise}.
    \end{cases}
\end{equation*}

We consider each case in turn.

\subsubsection{The case $\ov{d}=\ov{7}$}
In this case, $F=\mQ$, $T=\{P_{\infty}\}$, and $w=11$ clearly form a TPE for $C_{d}(\mQ)$, and hence
\begin{equation*}
    C_{d}(\mQ)\cap J(\mQ)_{\mathrm{tor}}=\{P_{\infty}\}.
\end{equation*}
In particular, if $\operatorname{rank}J(\mQ)=0$, then
\begin{equation*}
    C_{d}(\mQ)=\{P_{\infty}\}.
\end{equation*}

\subsubsection{The case $\ov{d}=\ov{9}$}
If $d$ is a perfect square, then $P_{\pm\sqrt{d}}\coloneqq(0,\pm\sqrt{d})\in C(\mQ)$ and $\operatorname{div}(y\mp\sqrt{d})=5(P_{\pm\sqrt{d}}-P_{\infty})$. Thus, we can take $F=\mQ,\ T=\{P_{\pm\sqrt{d}},P_{\infty}\},\ w=11$ to make up a TPE for $C_{d}(\mQ)$, and we obtain $C_{d}(\mQ)=\{ P_{\pm\sqrt{d}},P_{\infty} \}$. If $d$ is not a perfect square, let $F=\mQ(\sqrt{d})$, let $T=\{ P_{\pm\sqrt{d}}, P_{\infty}\}$, and let $w$ be a prime of $F$ lying above 11. Then these constitute a TPE for $C_{d}(\mQ)$. Indeed, since $\ov{d}=\ov{9}=\ov{3}^2$ in $\mF_{11}$, the prime 11 splits completely in $F$. Moreover, $\iota(T)\subseteq J(F)[5]$ and $\#T=\#\wi{C_{d}}(\mF_{11})=3$. Therefore, $C_{d}(\mQ)\subseteq T$; in particular,
\begin{equation*}
    C_{d}(\mQ)=\{P_{\infty}\}.
\end{equation*}

\subsubsection{The case $\ov{d}=\ov{1}$}
In this case, $\#\wi{C}_{d}(\mF_{11})=8$. Let
\begin{equation*}
    F=\mQ(\zeta_{5},\sqrt{d},\sqrt[5]{d}),
\end{equation*}
where $\zeta_{5}$ denotes a primitive fifth root of unity. Since $\ov{d}=\ov{1}$ in $\mF_{11}$, the prime $11$ splits completely in $F$. Choose a prime $w$ of $F$ lying above $11$, and define
\begin{equation*}
    T=\{ P_{\infty},P_{\pm\sqrt{d}}=(0,\pm\sqrt{d}),(-\zeta_{5}^{i}\sqrt[5]{d},0)\ \mid\ i=0,\dots,4 \}.
\end{equation*}
Then $\#T=8$. Hence $(F,T,w)$ is a TPE for $C_{d}(\mQ)$. By Theorem~\ref{MainThm}, we obtain
\begin{equation*}
    C_{d}(\mQ)\cap J(\mQ)_{\mathrm{tor}}\subseteq T.
\end{equation*}
In particular, if $\operatorname{rank}J(\mQ)=0$, then
\begin{equation*}
    C_{d}(\mQ)=
            \begin{cases}
                \{ (0,\pm\sqrt{d}),P_{\infty} \} & \text{if $d$ is a perfect square but not a fifth power}, \\
                \{(-\sqrt[5]{d},0),P_{\infty}\} & \text{if $d$ is not a perfect square but is a fifth power}, \\
                \{ P_{\infty} \} & \text{if $d$ is neither a perfect square nor a fifth power}.
            \end{cases}
\end{equation*}

\subsubsection{The remaining case}
In this case, the result \cite[see p.206-207]{Coleman1986} shows that 
\begin{equation*}
    \#(C_{d}(\ov{\mQ})\cap J(\ov{\mQ})_{\mathrm{tor}})=10.
\end{equation*}
On the other hand, if a TPE $(F,T,w)$ for $C_{d}(\mQ)$ existed, then $T\subseteq C_{d}(F)\cap J(F)_{\mathrm{tor}}$ have to satisfy
\begin{equation*}
    \#T=\#\wi{C}_{d}(\mF_{w})\geq 11\ (\ov{d}\neq \ov{1},\ov{7},\ov{9}).
\end{equation*}
However, $T\subseteq C_{d}(\ov{\mQ})\cap J(\ov{\mQ})_{\mathrm{tor}}$, thus we can not apply Theorem \ref{MainThm} on this case.

\subsection{The hyperelliptic curve $y^2=x^{p-1}+dx^{\frac{p-1}{2}}-1$}
Let $p$ be a prime. For $d\in p\mZ$, let $D_{d}$ be the hyperelliptic curve defined by
\begin{equation}
    y^2=f(x)\coloneqq x^{p-1}+dx^{\frac{p-1}{2}}-1.
\end{equation}
Its genus is $\frac{p-3}{2}$. The set $D_{d}(\mQ)$ contains the rational points $P_{\pm}$. We set $r_{d}\coloneqq \operatorname{rank}J_{D_{d}}(\mQ)$.
A computation of the discriminant of $f(x)$ yields
\begin{equation*}
    \left(\frac{p-1}{2}\right)^{p-1}\left(4+d^2\right)^{\frac{p-1}{2}}.
\end{equation*}
Since $p\mid d$, we have
\begin{equation*}
    \left(\frac{p-1}{2}\right)^{p-1}(4+d^2)^{\frac{p-1}{2}}\equiv 1 \not\equiv 0\ \pmod{p}
\end{equation*}
and thus, $D_{d}$ has good reduction at $p$.
Let $F$ be the splitting field of $f(x)$. Then
\begin{equation*}
    f(x)\equiv x^{p-1}-1\ \pmod{p},
\end{equation*}
which factors over $\mF_{p}$ as a product of linear polynomials. Therefore, $p$ splits completely in $F$. Choose a place of $F$ lying above $p$, and denote it by $w$. Define
\begin{equation*}
    T=\{ (\alpha,0)\in F\times F\ \mid\ f(\alpha)=0 \}\cup\{P_{\pm}\}.
\end{equation*}
Then $\#T=\#\wi{D_{d}}(\mF_{p})=p+1$. Hence, $(F,T,w)$ is a TPE for $D_{d}(\mQ)$, and we obtain
\begin{equation*}
    D_{d}(\mQ)\cap J(\mQ)_{\mathrm{tor}}\subseteq T.
\end{equation*}

\subsection{The hyperelliptic curve $y^2=x^p-x$ }
Let $p$ be a prime with $p\geq 5$. Let $C$ be the hyperelliptic curve $y^2=x^p-x$. Let $F$, $T$, and $w$ be as follows:
\begin{align*}
    &F=\mQ(\zeta_{p-1}), \\
    &T=\{(0,0)\}\cup\{ (\zeta_{p-1}^i,0) \mid i=1,\dots,p-1 \}\cup\{P_{\infty}\}, \\
    &w:\text{a place of $F$ lying above $p$}.
\end{align*}
Then $(F,T,w)$ is a TPE for $C(\mQ)$. We now verify this. 

Clearly, \eqref{eq:TPEa} is satisfied. Since $p\ge 5$, it is an odd prime. Moreover, since $p\equiv 1\ \pmod{p-1}$, the prime $p$ splits completely in $F$. Hence, condition \eqref{eq:TPEb} is also satisfied. Since $x^p-x$ factors over $\mFp$ as a product of $p$ distinct linear factors, the curve $C$ has good reduction at $p$. In particular, condition \eqref{eq:TPEc} is satisfied. 

Since every element of $T$ is a Weierstrass point of $C$, we have $\iota(T)\subseteq J(F)[2]$. Furthermore, as $\#T=\#\wi{C}(\mFp)=p+1$, conditions \eqref{eq:TPEd} and \eqref{eq:TPEe} are also satisfied. 

Therefore, $C(\mQ)\cap J(\mQ)_{\mathrm{tor}}\subseteq T$, and in particular,
\begin{equation*}
    C(\mQ)\cap J(\mQ)_{\mathrm{tor}}=\{ (0,0),(1,0),(-1,0),P_{\infty} \}.
\end{equation*}
If $\operatorname{rank}J(\mQ)=0$ (e.g., $p=5,13,17$), then
\begin{equation*}
    C(\mQ)=\{ (0,0),(1,0),(-1,0),P_{\infty} \}.
\end{equation*}

\subsection{The hyperelliptic curve $y^2=(x^2-1)(x^2-4)(x+3)$}
Let $C$ be the hyperelliptic curve $y^2=f(x)=(x^2-1)(x^2-4)(x+3)$. Let $F=\mQ(\sqrt{15})$ and take $p=7$. Since $7$ splits completely in $F$, choose a prime $w$ of $F$ lying above $7$. The curve $C$ has good reduction at $7$, and
\begin{equation*}
    \#\wi{C}(\mF_{7})=8.
\end{equation*}
Define
\begin{equation*}
    T=\{ (\pm 1,0),(\pm 2,0),(-3,0),P_{\infty} \}\cup\{ (3,\pm 4\sqrt{15}) \}.
\end{equation*}
Then $\#T=8$. Let $C'$ denote the twist of $C$ with respect to the extension $F/\mQ$. Then $C'$ is the hyperelliptic curve defined by $y^2=15f(x)$ which is the quadratic twist of $C$ by the quadratic extension $F/\mQ$. Let $\sigma$ be the generator of the Galois group of $F/\mQ$. It is easy to see that  $J_C(F)\otimes_\mZ\mQ\simeq (J_C(F)\otimes_\mZ\mQ)^{\sigma=+1}\oplus (J_C(F)\otimes_\mZ\mQ)^{\sigma=-1}$ as a $\mQ[{\rm Gal}(F/\mQ)]$-module. It follows that $(J_C(F)\otimes_\mZ\mQ)^{\sigma=+1}
=J_C(\mQ)\otimes_\mZ\mQ$ and 
$(J_C(F)\otimes_\mZ\mQ)^{\sigma=-1}
=J_{C'}(\mQ)\otimes_\mZ\mQ$. Thus, we have
\begin{equation*}
    \operatorname{rank}J_{C}(F)= \operatorname{rank}J_{C}(\mQ)+\operatorname{rank}J_{C'}(\mQ).
\end{equation*}
Computing the right-hand side using Magma, we find that $\operatorname{rank}J_{C}(F)=0$. Hence $J_{C}(F)=J_{C}(F)_{\mathrm{tor}}$, and therefore $\iota(T)\subseteq J_{C}(F)_{\mathrm{tor}}$. Thus $(F,T,w)$ is a TPE for $C(\mQ)$.
Consequently, by Theorem~\ref{MainThm} we obtain $C(\mQ)\subseteq T$, and hence
\begin{equation*}
    C(\mQ)=\{ (\pm 1,0),(\pm 2,0),(-3,0),P_{\infty} \}.
\end{equation*}

\end{document}